\documentclass[a4paper]{article}
\usepackage[T1]{fontenc}
\usepackage{lmodern}
\usepackage{amsfonts,amsmath,amssymb,amsthm}
\usepackage{graphicx} 
\usepackage{amsfonts}
\usepackage{amsmath,amssymb}
\usepackage{amsthm}
\usepackage{array}
\usepackage{csquotes}
\usepackage{enumitem}
\usepackage{graphicx}
\usepackage[dvipsnames]{xcolor}
\usepackage{mathtools}
\usepackage{aligned-overset}
\usepackage{type1cm}
\usepackage{bm}
\usepackage{caption}
\usepackage{multirow}
\usepackage[
    backend=biber,
    style=alphabetic,
    arxiv=abs,
    isbn=false,
    doi=false,
    url=false,
    giveninits=true,
    maxalphanames=8,
    minalphanames=4,
    maxnames=8,
    minnames=4,]{biblatex}
\DeclareFieldFormat[misc]{title}{{#1}}
\usepackage[unicode,bookmarksdepth=3,bookmarksnumbered,bookmarksopen,pdfencoding=auto]{hyperref}
\definecolor{citeblue}{RGB}{95,158,175}
\definecolor{linkpink}{RGB}{196,106,122}
\definecolor{urlgreen}{RGB}{110,175,85}

\hypersetup{
    colorlinks=true,
    linktocpage=true,
    citecolor=citeblue,
    linkcolor=linkpink,
    urlcolor=urlgreen
}
\usepackage[margin=30truemm]{geometry}
\usepackage{aliascnt}
\usepackage{cleveref}

\theoremstyle{plain}
\newtheorem{theorem}{Theorem}[section]
\newtheorem*{theoremintro}{Theorem}

\newaliascnt{proposition}{theorem}
\newtheorem{proposition}[proposition]{Proposition}
\aliascntresetthe{proposition}

\newaliascnt{lemma}{theorem}
\newtheorem{lemma}[lemma]{Lemma}
\aliascntresetthe{lemma}

\newaliascnt{corollary}{theorem}
\newtheorem{corollary}[corollary]{Corollary}
\aliascntresetthe{corollary}

\newaliascnt{conjecture}{theorem}

\aliascntresetthe{conjecture}

\theoremstyle{definition}
\newaliascnt{definition}{theorem}
\newtheorem{definition}[definition]{Definition}
\aliascntresetthe{definition}

\newaliascnt{example}{theorem}
\newtheorem{example}[example]{Example}
\aliascntresetthe{example}

\newaliascnt{remark}{theorem}
\newtheorem{remark}[remark]{Remark}
\aliascntresetthe{remark}
\crefname{section}{Section}{Sections}
\crefname{subsection}{Section}{Sections}
\crefname{theorem}{Theorem}{Theorems}
\crefname{definition}{Definition}{Definitions}
\crefname{proposition}{Proposition}{Propositions}
\crefname{lemma}{Lemma}{Lemmas}
\crefname{corollary}{Corollary}{Corollaries}
\crefname{example}{Example}{Examples}
\crefname{proof}{Proof}{Proof}
\crefname{remark}{Remark}{Remarks}
\crefname{equation}{Eq.}{Eq.}
\crefname{figure}{Figure}{Figures}
\crefname{table}{Table}{Tables}
\crefname{conjecture}{Conjecture}{Conjectures}

\numberwithin{equation}{section}

\newcommand{\id}{\mathrm{id}}

\allowdisplaybreaks
\title{Non-degeneracy of Hopf triplets}
\author{Tomoro Mochida\thanks{Mathematical Institute, Tohoku University, 6-3, Aramaki, Aza-Aoba, Aoba-ku, Sendai, 980-8578, Japan\\
\hspace*{1.8em}Email: \href{mailto:tomorou.mochida.r5@dc.tohoku.ac.jp}{\texttt{tomorou.mochida.r5@dc.tohoku.ac.jp}}}}
\date{}

\begin{document}

\maketitle

\begin{abstract}
    We study the relationship between semisimplicity and non-degeneracy for finite-dimensional Hopf triplets over a field of characteristic zero. The non-degeneracy condition appears in the construction of \(4\)-manifold invariants from Hopf triplets, and it was not known whether non-semisimple examples satisfying this condition exist. We prove that a Hopf triplet is non-degenerate if and only if it is semisimple.
\end{abstract}

\section{Introduction}

Hopf algebras play an important role in low-dimensional topology, providing algebraic data for the construction of many invariants of knots, \(3\)-manifolds, and \(4\)-manifolds, as well as for extensions of these invariants to topological quantum field theories. In this paper, we focus on the Kuperberg invariants of 3-manifolds~\cite{kuperberg1991involutory,kuperberg1996noninvolutory} and Chaidez, Cotler, and Cui's trisection invariants of 4-manifolds~\cite{chaidez20224,chaidez2023combed}.

Kuperberg first constructed invariants of \(3\)-manifolds using finite-dimensional involutory Hopf algebras~\cite{kuperberg1991involutory}, which are the semisimple ones in characteristic zero. He later extended this construction to invariants of framed \(3\)-manifolds using arbitrary finite-dimensional (possibly non-semisimple) Hopf algebras~\cite{kuperberg1996noninvolutory}. In these constructions, \(3\)-manifolds are represented by Heegaard diagrams, and the intersections of the upper and lower circles are encoded using the structure morphisms of Hopf algebras. Invariants constructed from non-semisimple algebraic data may detect
information that their semisimple counterparts cannot.
This possibility motivates the study of non-semisimple invariants.

We now turn to the \(4\)-dimensional case. Inspired by Kuperberg's work, Chaidez, Cotler, and Cui constructed invariants of \(4\)-manifolds from algebraic data called Hopf triplets in \cite{chaidez20224}. Their construction uses trisection diagrams, which are \(4\)-dimensional analogs of Heegaard diagrams, and follows a procedure similar to that used in the construction of the Kuperberg invariants. A Hopf triplet consists of three Hopf algebras and pairwise skew pairings satisfying certain compatibility conditions (\cref{def:hopftriplet}). In this construction, the Hopf triplet is assumed to be semisimple, meaning that all three Hopf algebras are semisimple. Later, in \cite{chaidez2023combed}, they developed a construction intended to apply to non-semisimple Hopf triplets. Their construction requires balanced Hopf triplets satisfying a non-degeneracy condition (\cref{def:nondegenerate}). They were, however, unable at the time to find any non-semisimple Hopf triplets satisfying this condition. This raises the question of whether non-degeneracy necessarily implies semisimplicity, which we address in this paper.

The non-degeneracy condition was introduced to ensure that the evaluations of the three standard stabilization trisection diagrams of \(D^4\) are nonzero, so that they can be used as normalization factors in defining the trisection invariant. The algebraic meaning of this condition is not immediately apparent from its original form. The identity in \cref{rem:nondegeneratecalc}, however, allows us to reformulate it in terms of integrals of Hopf algebras. This gives the condition an algebraic interpretation (\cref{rem:nondegenerateintegral}) and leads to the following result:

\begin{theoremintro}[\cref{thm:nondegvsss}]
    A Hopf triplet is non-degenerate if and only if all three Hopf algebras are semisimple.
\end{theoremintro}

The proof combines classical results in the theory of Hopf algebras with several auxiliary lemmas. In particular, it relies on the theory of integrals developed by Larson, Radford, and Sweedler, together with the Nichols--Zoeller freeness theorem.

The rest of the paper is organized as follows: In \cref{sec:integral}, we recall several fundamental results and prepare some lemmas for the proof of our main theorem. In \cref{sec:double}, we review generalized Drinfeld doubles and give a sufficient condition for them to be unimodular. In \cref{sec:hopftriplet}, we introduce Hopf triplets and their non-degeneracy condition and explain its interpretation in terms of integrals. We then prove the equivalence between non-degeneracy and semisimplicity of Hopf triplets.

Throughout the paper, we fix a field \(\Bbbk\) of characteristic zero and denote by \(\Bbbk^\times\) the multiplicative group of nonzero elements.
All Hopf algebras are assumed to be finite-dimensional over \(\Bbbk\), unless otherwise stated.

\subsection*{Acknowledgment}
The author would like to thank Serban Matei Mihalache and Yuji Terashima for helpful discussions.
This work is supported by JSPS KAKENHI Grant Number JP26KJ0522.

\subsection*{AI declaration}
ChatGPT-6 Astra (OpenAI) suggested \cref{lem:cointonsubhopf} and a strategy that simplified the author's original proof of \cref{thm:nondegvsss} by avoiding a case-by-case analysis. It was also used to proofread the manuscript. The author independently verified and finalized all mathematical arguments and takes full responsibility for the content of the manuscript.

\section{Integrals of Hopf algebras}\label{sec:integral}

In this section, we recall some fundamental results on integrals of Hopf algebras. Integrals play an important role in the theory of finite-dimensional Hopf algebras, and they are central to the proof of our main theorem. In fact, the proof relies heavily on the results presented in this section. A basic reference is \cite[Chapter~10]{radford2012hopf}.

Let \(H=(H,m,1,\Delta,\varepsilon,S)\) be a Hopf algebra over \(\Bbbk\). We use the Sweedler notation \(\Delta(x)=x_{(1)}\otimes x_{(2)}\) for \(x\in H\). We denote by \(G(H)\) the set of grouplike elements of \(H\), i.e., nonzero elements \(x\in H\) with \(\Delta(x)=x\otimes x\). Note that for the dual Hopf algebra \(H^*\), \(G(H^*)=\operatorname{Alg}_\Bbbk(H,\Bbbk)\).

\subsection{Actions}

Since \(H\) itself can be regarded as an \(H\)-bimodule under its multiplication, the transpose action defines both a left action and a right action of \(H\) on \(H^*\).

\begin{definition}
    A left action and a right action of \(H\) on \(H^*\) are defined as
    \begin{equation*}
        (a\rightharpoonup p)(x) \coloneqq p(xa)\quad\text{and}\quad (p\leftharpoonup a)(x) \coloneqq p(ax),
    \end{equation*}
    respectively, for \(a,x\in H\) and \(p\in H^*\). Dually, a left action and a right action of \(H^*\) on \(H\) are defined as
    \begin{equation*}
        p\rightharpoonup x \coloneqq x_{(1)}p(x_{(2)})\quad\text{and}\quad x\leftharpoonup p \coloneqq p(x_{(1)})x_{(2)},
    \end{equation*}
    respectively, for \(x\in H\) and \(p\in H^*\).
\end{definition}

\subsection{Integrals}

Let \(H\) be a Hopf algebra and \(\eta\in G(H^*)=\operatorname{Alg}_\Bbbk(H,\Bbbk)\). Set
\[
    L_\eta \coloneqq \{\Lambda\in H\mid x\Lambda=\eta(x)\Lambda \;\text{for all \(x\in H\)}\}\quad\text{and}\quad R_\eta \coloneqq \{\Lambda\in H\mid \Lambda x=\Lambda\eta(x) \;\text{for all \(x\in H\)}\}.
\]
It is shown in \cite[Proposition~3]{radford1976order} that \(L_\eta\) and \(R_\eta\) are one-dimensional (two-sided) ideals of \(H\).

\begin{definition}
    Let \(H\) be a Hopf algebra. A \emph{left} (resp. \emph{right}) \emph{integral} of \(H\) is an element of \(L_\varepsilon\) (resp. \(R_\varepsilon\)) associated with the counit \(\varepsilon\). A \emph{two-sided integral} is an element of both \(L_\varepsilon\) and \(R_\varepsilon\).
\end{definition}

Since the space \(L_\varepsilon\) is one-dimensional, there is a unique \(\chi\in G(H^*)\), called the \emph{distinguished grouplike element} of \(H^*\), such that 
\[
    \Lambda x = \Lambda\chi(x)
\]
for any \(\Lambda\in L_\varepsilon\) and \(x\in H\). Then \(L_\varepsilon=R_\chi\).

\begin{definition}
    A Hopf algebra \(H\) is \emph{unimodular} if \(L_\varepsilon=R_\varepsilon\). This is equivalent to \(\chi=\varepsilon\).
\end{definition}

We apply the same process to the dual Hopf algebra \(H^*\), but the distinguished grouplike element is defined using right integrals. That is, the distinguished grouplike element \(g\in G(H)\) of \(H\) is the unique element such that \(p\lambda=p(g)\lambda\) for any \(\lambda\in R_1\) and \(p\in H^*\), where \(R_1\) denotes the space of right integrals of \(H^*\), with \(1\in H\cong H^{**}\) identified with the counit of \(H^*\).

\begin{remark}
    Chaidez, Cotler, and Cui~\cite{chaidez2023combed} call the distinguished grouplike element of \(H\) the \emph{phase element} and denote it by \(a_H\), and call the distinguished grouplike element of \(H^*\) the \emph{cophase element} and denote it by \(\alpha_H\). In our notation, \(a_H=g_H\) and \(\alpha_H=\chi_H^{-1}\).
\end{remark}

We list several formulas involving the antipode, integrals, and distinguished grouplike elements.

\begin{lemma}\label{lem:integralss}
    Let \(H\) be a Hopf algebra, and let \(\Lambda\) be a nonzero left or right integral of \(H\). Then \(H\) is semisimple if and only if \(\varepsilon(\Lambda)\ne0\), in which case \(H\) is unimodular.
\end{lemma}

\begin{proof}
    See \cite[Propositions~3 and 4]{larson1969associative}.
\end{proof}

\begin{lemma}\label{lem:nonssalgmap}
    Let \(H\) be a non-semisimple Hopf algebra, and \(\eta,\xi\in G(H^*)\). Then \(\eta(L_\xi)=\eta(R_\xi)=0\).
\end{lemma}

\begin{proof}
    We show that \(\eta(L_\xi)=0\). \(\eta(R_\xi)=0\) is similar.
    
    We first consider the case \(\xi=\varepsilon\). That is, we prove that any left integral of a non-semisimple Hopf algebra is sent to \(0\) by any algebra map \(H\to\Bbbk\). To see this, we apply \(\eta\) to the defining identity of an integral
    \[
        x\Lambda = \varepsilon(x)\Lambda
    \]
    for \(x\in H\) and \(\Lambda\in L_\varepsilon\).
    If \(\eta(\Lambda)\ne0\), we have \(\eta=\varepsilon\), which contradicts \(\varepsilon(\Lambda)=0\), by \cref{lem:integralss}.
    
    Now we move on to the general case. Note that for \(\xi\in G(H^*)\), the endomorphism
    \[
        L(\xi)\colon H\longrightarrow H;\quad x\longmapsto L(\xi)(x)\coloneqq\xi\rightharpoonup x
    \]
    is an algebra map.
    By \cite[Lemma~2]{radford1976order}, we have
    \[
        L(\xi^{-1})(L_\varepsilon) = \xi^{-1}\rightharpoonup L_\varepsilon = L_\xi,
    \]
    so that any element of \(L_\xi\) is written as \(L(\xi^{-1})(\Lambda)\) for some left integral \(\Lambda\in L_\varepsilon\). Since \(\eta\circ L(\xi^{-1})\colon H\to\Bbbk\) is an algebra map, the result follows from the case \(\xi=\varepsilon\).
\end{proof}

\begin{lemma}\label{lem:integral_property}
    Let \(H\) be a Hopf algebra. Let \(\Lambda\) be a nonzero left integral of \(H\), let \(\lambda\) be a nonzero right integral of \(H^*\), and let \(\chi\) be the distinguished grouplike element of \(H^*\). Then
    \begin{enumerate}[label=\textup{(\arabic*)}, ref=\arabic*]
        \item \(\lambda(\Lambda)\ne 0\).\label{lem:integral_property1}
        \item \(S(\Lambda) = \chi\rightharpoonup\Lambda\) and \(S^{-1}(\Lambda) = \Lambda\leftharpoonup\chi\).\label{lem:integral_property2}
    \end{enumerate}
\end{lemma}

\begin{proof}
    See \cites[Proposition~7]{larson1969associative}[Corollary~1]{radford1990group} for part~(\ref{lem:integral_property1}), and \cite[Proposition~3]{radford1994trace} for part~(\ref{lem:integral_property2}).
\end{proof}

\begin{lemma}\label{lem:cointonsubhopf}
    Let \(U\) and \(V\) be Hopf subalgebras of a Hopf algebra \(H\), and let \(\lambda\) be a nonzero left or right integral of \(H^*\). If \(\lambda|_U\ne0\) and \(\lambda|_V\ne0\), then \(\lambda|_{U\cap V}\ne0\).
\end{lemma}

\begin{proof}
    Set \(M=H^*\), and consider the annihilators
    \[
        U^\perp\coloneqq\{p\in M\mid p(x)=0\;\text{for all \(x\in U\)}\}\quad\text{and}\quad V^\perp\coloneqq\{p\in M\mid p(y)=0\;\text{for all \(y\in V\)}\}
    \]
    of \(U\) and \(V\), respectively. They are two-sided ideals of \(M\). Since \((U\cap V)^\perp=U^\perp+V^\perp\), it now suffices to show that \(\lambda\notin U^\perp\) and \(\lambda\notin V^\perp\) imply \(\lambda\notin U^\perp+V^\perp\).

    Since the space of integrals is one-dimensional, there is \(h\in G(M^*)=\operatorname{Alg}_\Bbbk(M,\Bbbk)\) such that \(\lambda p = \lambda h(p)\) for all \(p\in M\). We fix a primitive idempotent \(e\in M\) such that \(h(e)=1\). Such an idempotent indeed exists by writing \(1\) as a sum of primitive orthogonal idempotents and applying \(h\). Then \(\lambda e=\lambda\).

    Since a finite-dimensional Hopf algebra is Frobenius~\cite[Remark]{larson1969associative}, the socle of the projective indecomposable left \(M\)-module \(Me\) is simple. Because \(\Bbbk\lambda\subseteq Me\) is a one-dimensional submodule, \(\operatorname{Soc}(Me)=\Bbbk\lambda\). Hence, any nonzero left submodule of \(Me\) contains \(\lambda\).

    Now \(U^\perp e\subseteq Me\) is a left submodule. If \(U^\perp e\ne0\), then \(\lambda\in U^\perp e\subseteq U^\perp\), which contradicts the assumption. So, we have \(U^\perp e=0\), and similarly, \(V^\perp e=0\). Therefore, \((U^\perp+V^\perp)e=0\), which implies \(\lambda=\lambda e\notin U^\perp+V^\perp\).
\end{proof}

\begin{lemma}\label{lem:intsubhopf}
    Let \(U\) and \(V\) be Hopf subalgebras of a Hopf algebra \(H\) such that \(K\coloneqq U\cap V\) is non-semisimple. Then for any left or right integral \(u\) of \(U\) and any left or right integral \(v\) of \(V\), we have \(uv=0\). 
\end{lemma}

\begin{proof}
    We may assume that \(u\ne0\) and \(v\ne0\). Since the space of integrals is one-dimensional, there are \(\eta,\xi\in G(K^*)=\operatorname{Alg}_\Bbbk(K,\Bbbk)\) such that
    \[
        ux = u\eta(x)\quad\text{and}\quad xv=\xi(x)v
    \]
    for all \(x\in K\).

    By the Nichols--Zoeller theorem~\cite[Theorem~7]{nichols1989hopf}, \(V\) is free as a left \(K\)-module. Let \(e_1,\ldots,e_n\in V\) be a basis of \(V\) as a left \(K\)-module, and write \(v\in V\) as
    \[
        v = \sum_{i=1}^n v_ie_i, \quad v_i\in K.
    \]
    It then follows from the identity \(xv=\xi(x)v\) for all \(x\in K\) that
    \[
        xv_i = \xi(x)v_i 
    \]
    for all \(x\in K\), and hence \(v_i\in L_\xi = \{\Lambda\in K\mid x\Lambda=\xi(x)\Lambda\;\text{for all \(x\in K\)}\}\).

    Now,
    \[
        uv = \sum_{i=1}^n uv_ie_i = \sum_{i=1}^nu\eta(v_i)e_i = 0,
    \]
    since \(\eta(v_i)=0\) by \cref{lem:nonssalgmap}.
\end{proof}

\section{Generalized Drinfeld doubles}\label{sec:double}

Quasitriangular Hopf algebras form an important class of Hopf algebras, as their close connection to the (quantum) Yang--Baxter equation gives rise to various invariants in low-dimensional topology. In this section, we review a slight generalization of the Drinfeld double construction.

\begin{definition}
    Let \(H_\alpha\) and \(H_\beta\) be Hopf algebras. A linear map \(\sigma\colon H_\alpha\otimes H_\beta\to\Bbbk\) is a \emph{skew pairing} if it satisfies the following conditions:
    \begin{enumerate}[label=(SP\arabic*),leftmargin=*,ref=SP\arabic*]
        \item \(\sigma(aa',b) = \sigma(a,b_{(1)})\sigma(a',b_{(2)})\),
        \item \(\sigma(a,bb') = \sigma(a_{(2)},b)\sigma(a_{(1)},b')\),
        \item \(\sigma(a,1_\beta) = \varepsilon_\alpha(a)\),
        \item \(\sigma(1_\alpha,b) = \varepsilon_\beta(b)\),
        \item \(\sigma(S_\alpha(a),b) = \sigma(a,S_\beta^{-1}(b))\)\label{def:skewpairing5}
    \end{enumerate}
    for all \(a,a'\in H_\alpha\) and \(b,b'\in H_\beta\). This is equivalent to saying that the induced map \(\sigma_l\colon H_\alpha\to H_\beta^{*\,\mathrm{cop}}\) is a Hopf algebra map, or equivalently, that \(\sigma_r\colon H_\beta\to H_\alpha^{*\,\mathrm{op}}\) is a Hopf algebra map.

    We call the triple \((H_\alpha,H_\beta,\sigma)\) a \emph{skew pair}.
\end{definition}

\begin{example}\label{ex:usualskewpair}
    The canonical evaluation pairing between a Hopf algebra \(H\) and its dual \(H^*\) defines a skew pairing \(\langle\cdot,\cdot\rangle\colon H^{*\,\mathrm{cop}}\otimes H\to\Bbbk\).
\end{example}

We can construct a new Hopf algebra out of a skew pair:

\begin{definition}
    Let \((H_\alpha,H_\beta,\sigma)\) be a skew pair. The \emph{generalized Drinfeld double} \(D_\sigma(H_\alpha,H_\beta)\) is a Hopf algebra defined as follows: As a coalgebra, \(D_\sigma(H_\alpha,H_\beta)=H_\alpha\otimes H_\beta\). The multiplication, unit, and antipode are defined as
    \begin{align*}
        (a\otimes b)(a'\otimes b') &\coloneqq \sigma(a'_{(1)},b_{(1)})\sigma(S_\alpha(a'_{(3)}),b_{(3)})aa'_{(2)}\otimes b_{(2)}b',\\
        1_{D_\sigma(H_\alpha,H_\beta)} &\coloneqq 1_\alpha\otimes 1_\beta,\\
        S(a\otimes b) &\coloneqq (1_\alpha\otimes S_\beta(b))(S_\alpha(a)\otimes 1_\beta),
    \end{align*}
    respectively, for \(a,a'\in H_\alpha\) and \(b,b'\in H_\beta\).
\end{definition}

\begin{example}
    For the skew pair \((H^{*\,\mathrm{cop}},H,\langle\cdot,\cdot\rangle)\) in \cref{ex:usualskewpair}, the corresponding generalized Drinfeld double is the usual Drinfeld double \(D(H)\) of \(H\).
\end{example}

\begin{lemma}[{\cite[Proposition~2.4]{doi1994multiplication}}]\label{lem:drinfeldmap}
    Let \((H_\alpha,H_\beta,\sigma)\) be a skew pair and \(H\) be a Hopf algebra. For any Hopf algebra maps \(f\colon H_\alpha\to H\) and \(g\colon H_\beta\to H\), the map
    \begin{equation*}
        \Phi\colon D_\sigma(H_\alpha,H_\beta)\longrightarrow H;\quad a\otimes b\longmapsto f(a)g(b)
    \end{equation*}
    is a Hopf algebra map if and only if \(f\) and \(g\) satisfy
    \begin{equation}\label{eq:drinfeldmap}
        \sigma(a_{(1)},b_{(1)})f(a_{(2)})g(b_{(2)})=g(b_{(1)})f(a_{(1)})\sigma(a_{(2)},b_{(2)})
    \end{equation}
    for all \(a\in H_\alpha\) and \(b\in H_\beta\).
\end{lemma}

\begin{proof}
    The map \(\Phi\) is a coalgebra map and preserves the unit. Thus it remains to check multiplicativity.
    We first show the `if' part. We calculate
    \begin{align*}
        \MoveEqLeft\Phi((a\otimes b)(a'\otimes b'))\\
        &= \Phi(\sigma(a'_{(1)},b_{(1)})\sigma(S_\alpha(a'_{(3)}),b_{(3)})aa'_{(2)}\otimes b_{(2)}b')\\
        &= \sigma(a'_{(1)},b_{(1)})\sigma(S_\alpha(a'_{(3)}),b_{(3)})f(a)f(a'_{(2)})g(b_{(2)})g(b')\\
        &= \sigma(a'_{(2)},b_{(2)})\sigma(S_\alpha(a'_{(3)}),b_{(3)})f(a)g(b_{(1)})f(a'_{(1)})g(b') && \text{(\cref{eq:drinfeldmap})}\\
        &= f(a)g(b)f(a')g(b') = \Phi(a\otimes b)\Phi(a'\otimes b').
    \end{align*}
    
    For the `only if' part, using the assumption that \(\Phi\) is an algebra map, we calculate \(\Phi((1_\alpha\otimes b_{(1)})(a_{(1)}\otimes 1_\beta)\sigma(a_{(2)},b_{(2)}))\) in two ways. On the one hand,
    \begin{align*}
        \MoveEqLeft\Phi((1_\alpha\otimes b_{(1)})(a_{(1)}\otimes 1_\beta)\sigma(a_{(2)},b_{(2)}))\\
        &= \Phi(1_\alpha\otimes b_{(1)})\Phi(a_{(1)}\otimes 1_\beta)\sigma(a_{(2)},b_{(2)})\\
        &= g(b_{(1)})f(a_{(1)})\sigma(a_{(2)},b_{(2)}),
        \intertext{and on the other hand,}
        \MoveEqLeft\Phi((1_\alpha\otimes b_{(1)})(a_{(1)}\otimes 1_\beta)\sigma(a_{(2)},b_{(2)}))\\
        &= \sigma(a_{(1)},b_{(1)})\Phi(a_{(2)}\otimes b_{(2)})\sigma(S_\alpha(a_{(3)}),b_{(3)})\sigma(a_{(4)},b_{(4)})\\
        &= \sigma(a_{(1)},b_{(1)})f(a_{(2)})g(b_{(2)}).\qedhere
    \end{align*}
\end{proof}

The next theorem generalizes the well-known fact that the Drinfeld double of a Hopf algebra is unimodular~\cites[Theorem~4]{radford1993minimal}[Proposition~7.4]{hennings1996invariants}:

\begin{theorem}\label{thm:drinfeldunimodular}
    Let \((H_\alpha,H_\beta,\sigma)\) be a skew pair. Let \(\Lambda_\alpha\) be a nonzero left integral of \(H_\alpha\), and let \(\Lambda'_\beta\) be a nonzero right integral of \(H_\beta\). Suppose that \(\sigma\) preserves the distinguished grouplike elements, i.e.,
    \begin{equation}\label{eq:skewcondition}
        \sigma(g_\alpha,b) = \chi_\beta(b)\quad\text{and}\quad \sigma(a,g_\beta) = \chi_\alpha^{-1}(a)
    \end{equation}
    for all \(a\in H_\alpha\) and \(b\in H_\beta\). Then the generalized Drinfeld double \(D_\sigma(H_\alpha,H_\beta)\) is unimodular with a two-sided integral \(\Lambda_\alpha\otimes\Lambda'_\beta\). In particular, the Drinfeld double of a Hopf algebra is unimodular.
\end{theorem}

\begin{proof}
    Temporarily forgetting the indices, let \(\Lambda\) and \(\Lambda'\) be left and right integrals of \(H\), respectively. Then we have two formulas
    \begin{align}
         \Lambda_{(1)}g^{-1}S(\Lambda_{(3)})\otimes\Lambda_{(2)} &= 1\otimes\Lambda,\label{eq:intdistinguishedprop1}\\
        S^{-1}(\Lambda'_{(3)})g^{-1}\Lambda'_{(1)}\otimes\Lambda'_{(2)} &= 1\otimes\Lambda'.\label{eq:intdistinguishedprop2}
    \end{align}
    Indeed, \cref{eq:intdistinguishedprop2} is due to \cite[Eq.~(15)]{radford1993minimal}. For \cref{eq:intdistinguishedprop1}, we substitute \(S(\Lambda)\), which is a right integral, into \cref{eq:intdistinguishedprop2} and apply \(\id_H\otimes S^{-1}\) to both sides.

    Then for \(a\otimes b\in D_\sigma(H_\alpha,H_\beta)\), the calculation
    \begin{align*}
        \MoveEqLeft(a\otimes b)(\Lambda_\alpha\otimes\Lambda'_\beta)\\
        &= \sigma(\Lambda_{\alpha(1)},b_{(1)})\sigma(S_\alpha(\Lambda_{\alpha(3)}),b_{(3)})a\Lambda_{\alpha(2)}\otimes b_{(2)}\Lambda'_\beta\\
        &= \sigma(\Lambda_{\alpha(1)},b_{(1)})\chi_\beta^{-1}(b_{(2)})\sigma(S_\alpha(\Lambda_{\alpha(3)}),b_{(3)})a\Lambda_{\alpha(2)}\otimes\Lambda'_\beta\\
        &= \sigma(\Lambda_{\alpha(1)},b_{(1)})\sigma(g_\alpha^{-1},b_{(2)})\sigma(S_\alpha(\Lambda_{\alpha(3)}),b_{(3)})a\Lambda_{\alpha(2)}\otimes\Lambda'_\beta &&\text{(\cref{eq:skewcondition})}\\
        &= \sigma(\Lambda_{\alpha(1)}g_\alpha^{-1}S_\alpha(\Lambda_{\alpha(3)}),b)a\Lambda_{\alpha(2)}\otimes\Lambda'_\beta\\
        &= \sigma(1,b)a\Lambda_\alpha\otimes\Lambda'_\beta &&\text{(\cref{eq:intdistinguishedprop1})}\\
        &= \varepsilon_\alpha(a)\varepsilon_\beta(b)\Lambda_\alpha\otimes\Lambda'_\beta
    \end{align*}
    shows that \(\Lambda_\alpha\otimes\Lambda'_\beta\) is a left integral, and the calculation
    \begin{align*}
        \MoveEqLeft(\Lambda_\alpha\otimes\Lambda'_\beta)(a\otimes b)\\
        &= \sigma(a_{(1)},\Lambda'_{\beta(1)})\sigma(a_{(3)},S_\beta^{-1}(\Lambda'_{\beta(3)}))\Lambda_\alpha a_{(2)}\otimes\Lambda'_{\beta(2)}b\\
        &= \sigma(a_{(1)},\Lambda'_{\beta(1)})\chi_\alpha(a_{(2)})\sigma(a_{(3)},S_\beta^{-1}(\Lambda'_{\beta(3)}))\Lambda_\alpha\otimes\Lambda'_{\beta(2)}b\\
        &= \sigma(a_{(1)},\Lambda'_{\beta(1)})\sigma(a_{(2)},g_\beta^{-1})\sigma(a_{(3)},S_\beta^{-1}(\Lambda'_{\beta(3)}))\Lambda_\alpha\otimes\Lambda'_{\beta(2)}b &&\text{(\cref{eq:skewcondition})}\\
        &= \sigma(a,S_\beta^{-1}(\Lambda'_{\beta(3)})g_\beta^{-1}\Lambda'_{\beta(1)})\Lambda_\alpha\otimes\Lambda'_{\beta(2)}b\\
        &= \sigma(a,1)\Lambda_\alpha\otimes\Lambda'_\beta b &&\text{(\cref{eq:intdistinguishedprop2})}\\
        &= \varepsilon_\alpha(a)\varepsilon_\beta(b)\Lambda_\alpha\otimes\Lambda'_\beta
    \end{align*}
    shows that \(\Lambda_\alpha\otimes\Lambda'_\beta\) is a right integral.
\end{proof}

\begin{corollary}
    Suppose that the hypotheses of \cref{thm:drinfeldunimodular} hold. Then \(D_\sigma(H_\alpha,H_\beta)\) is semisimple if and only if both \(H_\alpha\) and \(H_\beta\) are semisimple.
\end{corollary}

\begin{proof}
    Choose nonzero left and right integrals \(\Lambda_\alpha\in H_\alpha\) and \(\Lambda'_\beta\in H_\beta\), respectively. Then by \cref{thm:drinfeldunimodular}, \(\Lambda_\alpha\otimes\Lambda'_\beta\) is a nonzero integral of \(D_\sigma(H_\alpha,H_\beta)\). Since \(\varepsilon(\Lambda_\alpha\otimes\Lambda'_\beta)=\varepsilon_\alpha(\Lambda_\alpha)\varepsilon_\beta(\Lambda'_\beta)\), the corollary now follows from \cref{lem:integralss}.
\end{proof}

\begin{remark}
    It is also known that the canonical element of the Drinfeld double of a Hopf algebra is a universal \(R\)-matrix, and thus the Drinfeld double is quasitriangular. For a study of the quasitriangular structure of generalized Drinfeld doubles, see, for example, \cite{jiao2000skew}.
\end{remark}

\section{Hopf triplets}\label{sec:hopftriplet}

The notion of a Hopf triplet was first introduced in \cite{chaidez20224} for the construction of invariants of 4-manifolds. It involves three Hopf algebras together with pairwise skew pairings that satisfy certain conditions.

\subsection{Definition}

\begin{definition}\label{def:hopftriplet}
    A \emph{Hopf triplet} \(\mathcal{H}=(H_\alpha,H_\beta,H_\gamma,\sigma)\) consists of three Hopf algebras \(H_\alpha\), \(H_\beta\), and \(H_\gamma\), and three pairings \(\sigma_{\mu\nu}\colon H_\mu\otimes H_\nu\to\Bbbk\) (\(\mu\nu\in\{\alpha\beta,\beta\gamma,\gamma\alpha\}\)) that satisfy the following conditions:
    \begin{enumerate}[label=(HT\arabic*),leftmargin=*,ref=HT\arabic*]
        \item For each \(\mu\nu\in\{\alpha\beta,\beta\gamma,\gamma\alpha\}\), \(\sigma_{\mu\nu}\colon H_\mu\otimes H_\nu\to\Bbbk\) is a skew pairing.\label{def:hopftriplet1}
        \item Each \(\sigma_{\mu\nu}\) preserves the distinguished grouplike elements, as in \cref{eq:skewcondition}.
        \item The following \emph{triangle identity} holds:
        \begin{equation*}
            \sigma_{\alpha\beta}(a_{(1)},b_{(2)})\sigma_{\beta\gamma}(b_{(1)},c_{(2)})\sigma_{\gamma\alpha}(S_\gamma^2(c_{(1)}),a_{(2)}) = \sigma_{\alpha\beta}(a_{(2)},b_{(1)})\sigma_{\beta\gamma}(b_{(2)},c_{(1)})\sigma_{\gamma\alpha}(S_\gamma^2(c_{(2)}),a_{(1)})
        \end{equation*}
        for all \(a\in H_\alpha\), \(b\in H_\beta\), and \(c\in H_\gamma\).\label{def:hopftriplet3}
    \end{enumerate}
\end{definition}

Following \cite{chaidez2023combed}, the triangle identity is interpreted as follows: Let \(H_\alpha\), \(H_\beta\), and \(H_\gamma\) be Hopf algebras, and let \(\sigma_{\mu\nu}\colon H_\mu\otimes H_\nu\to\Bbbk\) (\(\mu\nu\in\{\alpha\beta,\beta\gamma,\gamma\alpha\}\)) be skew pairings. Then \(\sigma_{\gamma\alpha}\) and \(\sigma_{\beta\gamma}\) induce Hopf algebra maps 
\[
    \sigma_{\gamma\alpha,r}\colon H_\alpha^\mathrm{op}\to H_\gamma^*\quad\text{and}\quad \sigma_{\beta\gamma,l}\colon H_\beta^\mathrm{cop}\to H_\gamma^*,
\]
respectively. If the maps \((S_\gamma^2)^*\circ\sigma_{\gamma\alpha,r}\) and \(\sigma_{\beta\gamma,l}\) satisfy \cref{eq:drinfeldmap} in \cref{lem:drinfeldmap}, the map 
\[
    \Phi_{\alpha\beta}\colon D_{\sigma_{\alpha\beta}}(H_\alpha^\mathrm{op},H_\beta^\mathrm{cop})\longrightarrow H_\gamma^*;\quad a\otimes b\longmapsto ((S_\gamma^2)^*\circ\sigma_{\gamma\alpha,r})(a)\sigma_{\beta\gamma,l}(b)=\sigma_{\gamma\alpha}(S_\gamma^2(\cdot),a)\sigma_{\beta\gamma}(b,\cdot)
\]
is a Hopf algebra map. Evaluating \cref{eq:drinfeldmap} at \(c\in H_\gamma\) gives precisely the triangle identity. We have the following:

\begin{proposition}[{\cite[Lemma~3.4 and Theorem~3.6]{chaidez2023combed}}]
    Let \(H_\alpha\) ,\(H_\beta\), and \(H_\gamma\) be Hopf algebras, and let \(\sigma_{\mu\nu}\colon H_\mu\otimes H_\nu\to\Bbbk\) (\(\mu\nu\in\{\alpha\beta,\beta\gamma,\gamma\alpha\}\)) be skew pairings. Then the following are equivalent:
    \begin{enumerate}[label=\textup{(\arabic*)},ref=\arabic*]
        \item They satisfy the triangle identity. \label{prop:triangleeq1}
        \item \(\Phi_{\alpha\beta}\colon D_{\sigma_{\alpha\beta}}(H_\alpha^\mathrm{op},H_\beta^\mathrm{cop})\to H_\gamma^*\) is a Hopf algebra map. \label{prop:triangleeq2}
        \item They satisfy the following different version of the triangle identity:
        \begin{align*}
            \MoveEqLeft\sigma_{\alpha\beta}(S_\alpha(a_{(1)}),b_{(2)})\sigma_{\beta\gamma}(S_\beta(b_{(1)}),S_\gamma^2(c_{(2)}))\sigma_{\gamma\alpha}(S_\gamma(c_{(1)}),a_{(2)})\\
            &= \sigma_{\alpha\beta}(S_\alpha(a_{(2)}),b_{(1)})\sigma_{\beta\gamma}(S_\beta(b_{(2)}),S_\gamma^2(c_{(1)}))\sigma_{\gamma\alpha}(S_\gamma(c_{(2)}),a_{(1)})
        \end{align*}
        for all \(a\in H_\alpha\), \(b\in H_\beta\), and \(c\in H_\gamma\).\label{prop:triangleeq3}
    \end{enumerate}
    Using the compatibility of the pairings and the antipodes, we see that the triangle identity is symmetric with respect to the cyclic permutation of the indices \(\alpha\mapsto\beta\mapsto\gamma\mapsto\alpha\). Hence, the following conditions are also equivalent to (\ref{prop:triangleeq1})--(\ref{prop:triangleeq3}):
    \begin{enumerate}[label=\textup{(\arabic*)},ref=\arabic*]
        \setcounter{enumi}{3}
        \item \(\Phi_{\beta\gamma}\colon D_{\sigma_{\beta\gamma}}(H_\beta^\mathrm{op},H_\gamma^\mathrm{cop})\to H_\alpha^*\) is a Hopf algebra map.
        \item \(\Phi_{\gamma\alpha}\colon D_{\sigma_{\gamma\alpha}}(H_\gamma^\mathrm{op},H_\alpha^\mathrm{cop})\to H_\beta^*\) is a Hopf algebra map.
    \end{enumerate}
\end{proposition}

\begin{definition}
    A Hopf triplet \(\mathcal{H}=(H_\alpha,H_\beta,H_\gamma,\sigma)\) is \emph{semisimple} if all three Hopf algebras \(H_\alpha\), \(H_\beta\), and \(H_\gamma\) are semisimple. Otherwise, it is \emph{non-semisimple}.
\end{definition}

\subsection{Non-degeneracy of Hopf triplets}

The non-degeneracy condition was introduced in \cite{chaidez2023combed}.
Let \(\mathcal{H}=(H_\alpha,H_\beta,H_\gamma,\sigma)\) be a Hopf triplet. For \(\mu\nu\rho\in\{\alpha\beta\gamma,\beta\gamma\alpha,\gamma\alpha\beta\}\), define the linear map
\[
    \Psi_{\mu\nu}\colon D_{\sigma_{\mu\nu}}(H_\mu^\mathrm{op},H_\nu^\mathrm{cop})\longrightarrow H_\rho^*;\quad x\otimes y\longmapsto\sigma_{\rho\mu,r}(x)\sigma_{\nu\rho,l}(y)=\sigma_{\rho\mu}(\cdot,x)\sigma_{\nu\rho}(y,\cdot).
\]

\begin{definition}\label{def:nondegenerate}
    Let \(\mathcal{H}=(H_\alpha,H_\beta,H_\gamma,\sigma)\) be a Hopf triplet, and let \(\Lambda_\alpha\in H_\alpha\), \(\Lambda_\beta\in H_\beta\), and \(\Lambda_\gamma\in H_\gamma\) be nonzero left integrals. Then \(\mathcal{H}\) is \emph{non-degenerate} if 
    \begin{gather*}
        \Psi_{\alpha\beta}(S_\alpha(\Lambda_\alpha)\otimes S_\beta(\Lambda_\beta))(\Lambda_\gamma)\ne 0,\quad \Psi_{\beta\gamma}(S_\beta(\Lambda_\beta)\otimes S_\gamma(\Lambda_\gamma))(\Lambda_\alpha)\ne 0,\\
        \text{and}\quad \Psi_{\gamma\alpha}(S_\gamma(\Lambda_\gamma)\otimes S_\alpha(\Lambda_\alpha))(\Lambda_\beta)\ne 0.
    \end{gather*}
    Note that this condition is independent of the choice of nonzero integrals. If one replaces each \(\Lambda_\mu\) with \(c_\mu\Lambda_\mu\) (\(c_\mu\in\Bbbk^\times\)), each evaluation is multiplied by \(c_\alpha c_\beta c_\gamma\).
\end{definition}

\begin{remark}\label{rem:nondegeneratecalc}
    The original form of the definition of non-degeneracy given in \cite{chaidez2023combed} is
    \begin{equation*}
        \sigma_{\rho\mu}(S_\rho^{-1}(\Lambda_{\rho(1)}),\chi_\mu^{-1}\rightharpoonup\Lambda_\mu)\sigma_{\nu\rho}(\Lambda_\nu,\Lambda_{\rho(2)})\ne0,\quad \mu\nu\rho\in\{\alpha\beta\gamma,\beta\gamma\alpha,\gamma\alpha\beta\}
    \end{equation*}
    and looks different from ours. We can, however, calculate
    \begin{align*}
        \MoveEqLeft\sigma_{\rho\mu}(S_\rho^{-1}(\Lambda_{\rho(1)}),\chi_\mu^{-1}\rightharpoonup\Lambda_\mu)\sigma_{\nu\rho}(\Lambda_\nu,\Lambda_{\rho(2)})\\
        &= \sigma_{\rho\mu}(S_\rho^{-1}(\Lambda_{\rho(1)}),\chi_\mu^{-1}\rightharpoonup\Lambda_\mu)\sigma_{\nu\rho}(S_\nu^{-1}(\Lambda_\nu),S_\rho^{-1}(\Lambda_{\rho(2)}))\\
        &= \sigma_{\rho\mu}(S_\rho^{-1}(\Lambda_\rho)_{(2)},\chi_\mu^{-1}\rightharpoonup\Lambda_\mu)\sigma_{\nu\rho}(\Lambda_\nu\leftharpoonup\chi_\nu,S_\rho^{-1}(\Lambda_\rho)_{(1)})\\
        &= \sigma_{\rho\mu}(S_\rho^{-1}(\Lambda_\rho)_{(2)},\Lambda_{\mu(1)}\chi_\mu^{-1}(\Lambda_{\mu(2)}))\sigma_{\nu\rho}(\chi_\nu(\Lambda_{\nu(1)})\Lambda_{\nu(2)},S_\rho^{-1}(\Lambda_\rho)_{(1)})\\
        &= \sigma_{\rho\mu}(S_\rho^{-1}(\Lambda_\rho)_{(2)},\Lambda_{\mu(1)})\sigma_{\rho\mu}(g_\rho^{-1},\Lambda_{\mu(2)})\sigma_{\nu\rho}(\Lambda_{\nu(2)},S_\rho^{-1}(\Lambda_\rho)_{(1)})\sigma_{\nu\rho}(\Lambda_{\nu(1)},g_\rho^{-1})\\
        &= \sigma_{\rho\mu}(S_\rho^{-1}(\Lambda_\rho)_{(2)}g_\rho^{-1},\Lambda_\mu)\sigma_{\nu\rho}(\Lambda_\nu,S_\rho^{-1}(\Lambda_\rho)_{(1)}g_\rho^{-1})\\
        &= \sigma_{\rho\mu}((S_\rho^{-1}(\Lambda_\rho)g_\rho^{-1})_{(2)},\Lambda_\mu)\sigma_{\nu\rho}(\Lambda_\nu,(S_\rho^{-1}(\Lambda_\rho)g_\rho^{-1})_{(1)})\\
        &= \varepsilon_\rho(g_\rho^{-1})\sigma_{\rho\mu}(S_\rho^{-1}(\Lambda_\rho)_{(2)},\Lambda_\mu)\sigma_{\nu\rho}(\Lambda_\nu,S_\rho^{-1}(\Lambda_\rho)_{(1)})\\
        &= \sigma_{\rho\mu}(\Lambda_{\rho{(1)}},S_\mu(\Lambda_\mu))\sigma_{\nu\rho}(S_{\nu}(\Lambda_\nu),\Lambda_{\rho(2)}),
    \end{align*}
    arriving at the same definition. Here, we use \cref{lem:integral_property}~(\ref{lem:integral_property2}) in the second equality, \cref{eq:skewcondition} in the fourth equality, and the fact that \(S_\rho^{-1}(\Lambda_\rho)\) is a right integral in the seventh equality.
    This identity is the key, allowing us to formulate the non-degeneracy condition in terms of integrals, as explained in the following remark.
\end{remark}

\begin{remark}\label{rem:nondegenerateintegral}
    For each \(\mu\nu\rho\in\{\alpha\beta\gamma,\beta\gamma\alpha,\gamma\alpha\beta\}\), \(S_\mu(\Lambda_\mu)\) is a left integral of \(H_\mu^\mathrm{op}\), and \(S_\nu(\Lambda_\nu)\) is a right integral of \(H_\nu^\mathrm{cop}\). Thus, by \cref{thm:drinfeldunimodular}, \(S_\mu(\Lambda_\mu)\otimes S_\nu(\Lambda_\nu)\) is a two-sided integral of the generalized Drinfeld double \(D_{\sigma_{\mu\nu}}(H_\mu^\mathrm{op},H_\nu^\mathrm{cop})\). The non-degeneracy condition asserts that the image of the integral \(S_\mu(\Lambda_\mu)\otimes S_\nu(\Lambda_\nu)\in D_{\sigma_{\mu\nu}}(H_\mu^\mathrm{op},H_\nu^\mathrm{cop})\) under \(\Psi_{\mu\nu}\) takes a nonzero value on the integral \(\Lambda_\rho\in H_\rho\).
\end{remark}

The following is our main theorem:

\begin{theorem}\label{thm:nondegvsss}
    A Hopf triplet is non-degenerate if and only if it is semisimple.
\end{theorem}

\begin{proof}
    For the `if' part, since a finite-dimensional semisimple Hopf algebra over a field of characteristic \(0\) is involutory, i.e., \(S^2=\id\) \cite[Theorem~4]{larson1988semisimple}, we have \(\Psi_{\mu\nu} = \Phi_{\mu\nu}\) (\(\mu\nu\in\{\alpha\beta,\beta\gamma,\gamma\alpha\}\)), and hence each \(\Psi_{\mu\nu}\) is a Hopf algebra map.

    For \(\mu\nu\rho\in\{\alpha\beta\gamma,\beta\gamma\alpha,\gamma\alpha\beta\}\), set the Hopf subalgebra \(H\coloneqq \Psi_{\mu\nu}(D_{\sigma_{\mu\nu}}(H_\mu^\mathrm{op},H_\nu^\mathrm{cop}))\subseteq H_\rho^*\). Then \(\Psi_{\mu\nu}(S_\mu(\Lambda_\mu)\otimes S_\nu(\Lambda_\nu))\) is a nonzero (two-sided) integral of \(H\), and \(\Lambda_\rho|_H\) is a nonzero (two-sided) integral of \(H^*\). Indeed, 
    \[
        \varepsilon_H(\Psi_{\mu\nu}(S_\mu(\Lambda_\mu)\otimes S_\nu(\Lambda_\nu)))=\varepsilon_\mu(\Lambda_\mu)\varepsilon_\nu(\Lambda_\nu)\ne0,\quad\text{and}\quad(\Lambda_\rho|_H)(1_H)=\varepsilon_\rho(\Lambda_\rho)\ne0
    \]
    by \cref{lem:integralss}. Hence, by \cref{lem:integral_property}~(\ref{lem:integral_property1}), \(\Psi_{\mu\nu}(S_\mu(\Lambda_\mu)\otimes S_\nu(\Lambda_\nu))(\Lambda_\rho)\ne0\).

    For the `only if' part, set
    \begin{align*}
        E_\gamma &\coloneqq \Psi_{\alpha\beta}(S_\alpha(\Lambda_\alpha)\otimes S_\beta(\Lambda_\beta))(\Lambda_\gamma)=\sigma_{\gamma\alpha}(\Lambda_{\gamma(1)},S_\alpha(\Lambda_\alpha))\sigma_{\beta\gamma}(S_\beta(\Lambda_\beta),\Lambda_{\gamma(2)}),\\
        E_\alpha &\coloneqq \Psi_{\beta\gamma}(S_\beta(\Lambda_\beta)\otimes S_\gamma(\Lambda_\gamma))(\Lambda_\alpha)=\sigma_{\alpha\beta}(\Lambda_{\alpha(1)},S_\beta(\Lambda_\beta))\sigma_{\gamma\alpha}(S_\gamma(\Lambda_\gamma),\Lambda_{\alpha(2)}),\\
        E_\beta &\coloneqq \Psi_{\gamma\alpha}(S_\gamma(\Lambda_\gamma)\otimes S_\alpha(\Lambda_\alpha))(\Lambda_\beta)=\sigma_{\beta\gamma}(\Lambda_{\beta(1)},S_\gamma(\Lambda_\gamma))\sigma_{\alpha\beta}(S_\alpha(\Lambda_\alpha),\Lambda_{\beta(2)}).
    \end{align*}
    We show that for \(\mu\nu\rho\in\{\alpha\beta\gamma,\beta\gamma\alpha,\gamma\alpha\beta\}\), if \(H_\rho\) is non-semisimple, \(E_\mu\ne0\), and \(E_\nu\ne0\), then \(E_\rho=0\).

    Set
    \[
        H\coloneqq H_\rho^*,\quad\lambda\coloneqq \Lambda_\rho\in H^*,\quad U\coloneqq \sigma_{\rho\mu,r}(H_\mu^{\mathrm{op}})\subseteq H,\quad\text{and}\quad V\coloneqq \sigma_{\nu\rho,l}(H_\nu^{\mathrm{cop}})\subseteq H.
    \]
    By (\ref{def:skewpairing5}), \(E_\mu\ne0\) implies \(\lambda|_U\ne0\), and \(E_\nu\ne0\) implies \(\lambda|_V\ne0\). Thus, by \cref{lem:cointonsubhopf}, \(\lambda|_K\ne0\) (\(K\coloneqq U\cap V\)). Since \(\lambda|_K\) is a nonzero integral of \(K^*\) and \(H_\rho\) is non-semisimple,
    \[
        \lambda(1_K)=\varepsilon_\rho(\Lambda_\rho)=0,
    \]
    so that \(K^*\) is non-semisimple by \cref{lem:integralss}. Since the ground field \(\Bbbk\) has characteristic zero, \(K\) is also non-semisimple~\cite[Theorem~3.3]{larson1988finite}. From the observations that \(u\coloneqq \sigma_{\rho\mu,r}(S_\mu(\Lambda_\mu))\in U\) is a left integral of \(U\), and that \(v\coloneqq \sigma_{\nu\rho,l}(S_\nu(\Lambda_\nu))\in V\) is a right integral of \(V\), we have \(uv=0\) by \cref{lem:intsubhopf}. Hence,
    \[
        E_\rho=\lambda(uv)=0,
    \]
    as desired. This completes the proof. We note that this part does not use the triangle identity~(\ref{def:hopftriplet3}).
\end{proof}

\begin{remark}
    Finally, we comment on the semisimple trisection invariant of \cite{chaidez20224}. In their construction, they impose the admissibility condition that the evaluation of the standard genus-\(3\) trisection diagram of \(S^4\) admits a cube root in \(\Bbbk^\times\). It is easy to see that this evaluation is equal to \(E_\alpha E_\beta E_\gamma\), which is nonzero in the semisimple setting by \cref{thm:nondegvsss}. Thus, if \(\Bbbk\) is algebraically closed, this admissibility condition is automatically satisfied.
\end{remark}

\printbibliography[heading=bibintoc]
\end{document}